\documentclass[12pt]{amsart}
\usepackage{amsmath,amssymb,amsbsy,amsfonts,amsthm,latexsym,
                        amsopn,amstext,amsxtra,euscript,amscd,mathrsfs,color,bm}
                        
\usepackage{float} 
\usepackage[english]{babel}
\usepackage{url}
\usepackage[colorlinks,linkcolor=blue,anchorcolor=blue,citecolor=blue]{hyperref}

\begin{document}

\newtheorem{theorem}{Theorem}
\newtheorem{lemma}[theorem]{Lemma}
\newtheorem{example}[theorem]{Example}
\newtheorem{algol}{Algorithm}
\newtheorem{corollary}[theorem]{Corollary}
\newtheorem{prop}[theorem]{Proposition}
\newtheorem{definition}[theorem]{Definition}
\newtheorem{question}[theorem]{Question}
\newtheorem{problem}[theorem]{Problem}
\newtheorem{remark}[theorem]{Remark}
\newtheorem{conjecture}[theorem]{Conjecture}

\newcommand{\comm}[1]{\marginpar{%
\vskip-\baselineskip 
\raggedright\footnotesize
\itshape\hrule\smallskip#1\par\smallskip\hrule}}

\newcommand{\commI}[1]{\marginpar{%
\begin{color}{blue}
\vskip-\baselineskip 
\raggedright\footnotesize
\itshape\hrule \smallskip I: #1\par\smallskip\hrule\end{color}}}


\def\cA{{\mathcal A}}
\def\cB{{\mathcal B}}
\def\cC{{\mathcal C}}
\def\cD{{\mathcal D}}
\def\cE{{\mathcal E}}
\def\cF{{\mathcal F}}
\def\cG{{\mathcal G}}
\def\cH{{\mathcal H}}
\def\cI{{\mathcal I}}
\def\cJ{{\mathcal J}}
\def\cK{{\mathcal K}}
\def\cL{{\mathcal L}}
\def\cM{{\mathcal M}}
\def\cN{{\mathcal N}}
\def\cO{{\mathcal O}}
\def\cP{{\mathcal P}}
\def\cQ{{\mathcal Q}}
\def\cR{{\mathcal R}}
\def\cS{{\mathcal S}}
\def\cT{{\mathcal T}}
\def\cU{{\mathcal U}}
\def\cV{{\mathcal V}}
\def\cW{{\mathcal W}}
\def\cX{{\mathcal X}}
\def\cY{{\mathcal Y}}
\def\cZ{{\mathcal Z}}

\def\C{\mathbb{C}}
\def\F{\mathbb{F}}
\def\K{\mathbb{K}}
\def\G{\mathbb{G}}
\def\Z{\mathbb{Z}}
\def\R{\mathbb{R}}
\def\Q{\mathbb{Q}}
\def\N{\mathbb{N}}
\def\M{\textsf{M}}
\def\PP{\mathbb{P}}
\def\A{\mathbb{A}}
\def\p{\mathfrak{p}}
\def\n{\mathfrak{n}}
\def\X{\mathcal{X}}
\def\x{\textrm{\bf x}}
\def\w{\textrm{\bf w}}
\def\ovQ{\overline{\Q}}
\def\rank#1{\mathrm{rank}#1}
\def\wf{\widetilde{f}}
\def\wg{\widetilde{g}}

\def\({\left(}
\def\){\right)}
\def\[{\left[}
\def\]{\right]}
\def\<{\langle}
\def\>{\rangle}

\def\gen#1{{\left\langle#1\right\rangle}}
\def\genp#1{{\left\langle#1\right\rangle}_p}
\def\genPs{{\left\langle P_1, \ldots, P_s\right\rangle}}
\def\genPsp{{\left\langle P_1, \ldots, P_s\right\rangle}_p}

\def\Mult{M} 

\def\e{e}

\def\eq{\e_q}
\def\fh{{\mathfrak h}}

\def\lcm{{\mathrm{lcm}}\,}

\def\({\left(}
\def\){\right)}
\def\fl#1{\left\lfloor#1\right\rfloor}
\def\rf#1{\left\lceil#1\right\rceil}
\def\mand{\qquad\mbox{and}\qquad}

\def\jt{\tilde\jmath}
\def\ellmax{\ell_{\rm max}}
\def\llog{\log\log}

\def\m{{\rm m}}
\def\ch{\hat{h}}
\def\GL{{\rm GL}}
\def\Orb{\mathrm{Orb}}
\def \S{\mathcal{S}}
\def\vec#1{\mathbf{#1}}
\def\ov#1{{\overline{#1}}}
\def\Gal{{\rm Gal}}

\newcommand{\bfalpha}{{\boldsymbol{\alpha}}}
\newcommand{\bfomega}{{\boldsymbol{\omega}}}

\newcommand{\Ch}{{\operatorname{Ch}}}
\newcommand{\Elim}{{\operatorname{Elim}}}
\newcommand{\proj}{{\operatorname{proj}}}
\newcommand{\h}{{\operatorname{h}}}

\newcommand{\hh}{\mathrm{h}}
\newcommand{\aff}{\mathrm{aff}}
\newcommand{\Spec}{{\operatorname{Spec}}}
\newcommand{\Res}{{\operatorname{Res}}}

\numberwithin{equation}{section}
\numberwithin{theorem}{section}

\title{On some extensions of the Ailon-Rudnick Theorem}

\author{Alina Ostafe}
\address{School of Mathematics and Statistics, University of New South Wales, Sydney NSW 2052, Australia}
\email{alina.ostafe@unsw.edu.au}


\begin{abstract} 
In this paper we present some extensions of the Ailon-Rudnick Theorem, which says that if $f,g\in\C[T]$, then $\gcd(f^n-1,g^m-1)$ is bounded for all $n,m\ge 1$. More precisely, using a uniform bound for the number of torsion points on curves and results on the intersection of curves with algebraic subgroups of codimension at least $2$, we present two such extensions in the univariate case. We also give two multivariate analogues of the Ailon-Rudnick Theorem based on Hilbert's irreducibility theorem and a result of Granville and Rudnick about torsion points on hypersurfaces.
\end{abstract}

\maketitle

\section{Introduction}

\subsection{Motivation}

Let $a,b$ be multiplicatively independent positive integers and $\varepsilon>0$. 
Bugeaud, Corvaja and Zannier~\cite{BCZ} have proved that
$$
\gcd\(a^n-1,b^n-1\)\le \exp(\varepsilon n)
$$
as $n$ tends to infinity. 
 Corvaja and Zannier~\cite{CZ05} have generalised  this result
 and replaced $a^n,b^n$ with 
 multiplicatively independent  $S$-units $u,v\in\Z$.

In the function field case, Ailon and Rudnick~\cite[Theorem~1]{AR} proved that if $f,g\in\C[T]$ are multiplicatively independent polynomials, then there exists $h \in \C[T]$ such 
that
\begin{equation}
\label{eq:AR}
\gcd(f^n-1,g^n-1) \mid h
\end{equation}
for all $n\ge 1$. Examining their argument one can easily see that the same
statement holds in a larger generality; namely there exists $\widetilde h \in \C[T]$
 such 
that
\begin{equation}
\label{eq:AR Gen}
\gcd(f^n-1,g^m-1) \mid \widetilde h
\end{equation}
for all $n,m\ge 1$.

In the case of finite fields $\F_q$ of characteristic $p$, Silverman~\cite{Si} proves that even more restrictions on the polynomials $f,g\in\F_q[T]$ does not allow a similar conclusion as the result of~\cite{AR}. In particular, Silverman proves that
the analogue of~\eqref{eq:AR} is false in a very strong sense:  there exists a constant $c(f,g;q)$, depending only on $f$,  $g$ and $q$, such that
$$
\deg \gcd(f^n-1,g^n-1)\ge c(f,g;q) n
$$
for infinitely many $n$.

More results in positive characteristic are obtained in~\cite{CZ13,Lau}, as well as variants for elliptic divisibility sequences~\cite{Si04,Si05}.

In this paper we present some extensions of the Ailon-Rudnick Theorem~\cite[Theorem 1]{AR} over $\C$, both in the univariate and multivariate cases.  
Although the method of proof in the univariate case is similar to, or reduces to using,~\cite{AR}, we find these extensions exciting and we hope they will be of independent interest. Moreover, as we explain below, in certain situation we reduce our problem to applying~\cite[Theorem 1]{AR}, however for this we need a uniform bound for~\eqref{eq:AR} that depends only on the degree of the polynomials $f$ and $g$. 

Besides the generality of results, the new ingredients of the paper are employing  results~\cite{BS,BMZ99,BMZ08,Mau}  on the number of points on intersections of 
curves in the $n$-dimensional  multiplicative torus $\G_m^n$ with algebraic subgroups. 
We also present two multivariate generalisations that are based on the use of  Hilbert's irreducibility theorem~\cite{Sch} and a transformation using the Kronecker substitution  to reduce the problem to the univariate case, as well as a result of Granville and Rudnick~\cite{GranRud} about torsion points on hypersurfaces.

\subsection{Conventions and notation} As usual, we denote $\C[X_1,\ldots,X_{\ell}]$ the polynomial ring in ${\ell}$ variables and $\C(X_1,\ldots,X_{\ell})$ the field of rational functions $F/G$, $F,G\in\C[X_1,\ldots,X_{\ell}]$. When working with univariate polynomials we reserve the variable $T$. All polynomials in $\C[T]$ are denoted with small letters $f,g,\ldots$, and for polynomials in $\C[X_1,\ldots,X_\ell]$ we use capital letters $F,G,\ldots$. 

Throughout the paper, for a univariate polynomial $f\in\C[T]$, the notation $d_f$ will be used for the degree of $f$. 

For a 
family of polynomials $F_1,\ldots,F_s \in \C[X_1, \ldots, X_{\ell}]$, we
denote by $Z(F_1,\ldots,F_s)$ their zero set in $\C^{m}$.

Throughout the paper we assume that the greatest common divisor of two (or more) 
polynomials is monic, so it  is well-defined.

We also define here the main concept of this paper.

\begin{definition}
The polynomials $F_1,\ldots,F_{s}\in\C[X_1,\ldots,X_{\ell}]$ are {\it multiplicatively independent} if there exists no nonzero vector $(\nu_1,\ldots,\nu_{s})\in\Z^{s}$ such that
$$
F_1^{\nu_1}\cdots F_s^{\nu_{s}}=1.
$$
Similarly, we say that the polynomials $F_1,\ldots,F_{s}\in\C[X_1,\ldots,X_{\ell}]$ are {\it multiplicatively independent in the group $\C(X_1,\ldots,X_{\ell})^*/\C^*$} if there exists no nonzero vector $(\nu_1,\ldots,\nu_{s})\in\Z^{s}$ and $a\in\C^*$ such that
$$
F_1^{\nu_1}\cdots F_s^{\nu_{s}}=a.
$$
\end{definition}

We present now in more details the main results of this paper.

\subsection{Our results: univariate case} 

Section~\ref{sec:prel} is dedicated to outlining the tools and results needed along the paper.
In particular, in Section~\ref{sec:ART} we recall the result of~\cite[Theorem~1]{AR} and, using a uniform bound for the number of points on a curve with coordinates roots of unity due to Beukers and Smyth~\cite{BS}, we derive in Lemma~\ref{lem:univAR} a version 
of~\eqref{eq:AR Gen} that gives an upper bound on $\deg \gcd(f^n-1,g^m-1)$
that depends only the degrees of $f$ and $g$ (rather than on 
the polynomials themselves).

Such a uniform bound is crucial for some of our main results presented below and proved  
in  Section~\ref{sec:main}. 
In particular, 
our first extension of~\cite[Theorem 1]{AR}, which  is proved in Section~\ref{sec:univar1}, is based on this uniform bound.
\begin{theorem}
\label{thm:genAR1}
Let $f,g,h_1,h_2\in\C[T]$. If $f$ and $g$ are multiplicatively independent in $\C(T)^*/\C^*$, then for all $n,m\ge 1$ we have
$$ 
\deg \gcd\(h_1\(f^n\),h_2\(g^m\)\)\le  d_{h_1}d_{h_2}  \(11(d_f+d_g)^{2}\)^{\min(d_f,d_g)}.
$$
\end{theorem}

 For the second
 extension of~\cite[Theorem~1]{AR}, which is proved in Section~\ref{sec:univar2}, we apply the finiteness result of~\cite{BMZ08,Mau}, see also~\cite{BMZ99}, for the number of points on the intersection of curves in $\G_m^n$ with algebraic subgroups, see Lemma~\ref{lem:bmz}. No uniform bounds are known so far for such finiteness results.

We recall that for a polynomial $f\in\C[T]$, we denote by $Z(f)$ the set of zeros of $f$ in $\C$.

 \begin{theorem}
\label{thm:genAR2}
Let $f_1,\ldots,f_{\ell},\varphi_1,\ldots,\varphi_{k},g_1,\ldots,g_r,\psi_1,\ldots,\psi_s\in\C[T]$, $\ell,k,r,s\ge 1$, be multiplicatively independent polynomials such that
\begin{equation}
\label{eq:cond}
Z(f_1\cdots f_{\ell})\cap Z(\varphi_1\cdots \varphi_k)=\emptyset,\quad Z(g_1\cdots g_r)\cap Z(\psi_1\cdots \psi_s)=\emptyset.
\end{equation} 
Then we have:
\begin{itemize}
\item[\bf{i.}]  For all $n_1,\ldots,n_{\ell},\nu_1,\ldots,\nu_k,m_1,\ldots,m_r,\mu_1,\ldots,\mu_s\ge 0$, there 
exists a polynomial $h\in \C[T]$ 
such that
$$
 \gcd\(\prod_{i=1}^{\ell}f_i^{n_i}-\prod_{i=1}^k\varphi_i^{\nu_i},\prod_{i=1}^rg_i^{m_i}-\prod_{i=1}^s\psi_i^{\mu_i}\)\mid h. 
$$

\item[\bf{ii.}]  If in addition
$$
\gcd(f_1\cdots f_{\ell}-1,g_1\cdots g_r-1)=1,
$$
then there exists a finite set $S$ and monoids $\cL_t \subseteq \N^{\ell+k+r+s}$, $t  \in S$, 
such that the remaining set 
$$
\cN = \N^{\ell+k+r+s}\setminus \cup_{t\in S} \cL_t
$$
is of positive 
asymptotic density and for any vector 
$$
(n_1,\ldots,n_{\ell},\nu_1,\ldots,\nu_k,m_1,\ldots,m_r,\mu_1,\ldots,\mu_s)\in\cN
$$ 
we have 
$$
\gcd\(\prod_{i=1}^{\ell}f_i^{n_i}-\prod_{i=1}^k\varphi_i^{\nu_i},\prod_{i=1}^rg_i^{m_i}-\prod_{i=1}^s\psi_i^{\mu_i}\)=1.
$$
\end{itemize}
\end{theorem}

Although we prefer to keep the language of polynomials, one can easily see that  
Theorem~\ref{thm:genAR2}  can be reformulated in terms of $S$-units in $\C[T]$ and implies  that for any set of $S$-units, 
there exists a polynomial $h\in \C[T]$ such that for any multiplicatively independent $S$-units $U,V$
we have 
$$
 \gcd\(U-1,V-1\)\mid h. 
$$
In particular, this extension of~\cite{AR} is fully analogous to the aforementioned extension 
of~\cite{CZ05}  over~\cite{BCZ}. 

We also compare Theorem~\ref{thm:genAR2}, which for multiplicatively independent $S$-units $U,V$, gives  a uniform bound for $\deg\gcd(U-1,V-1)$, while the result of Corvaja and Zannier~\cite[Corollary 2.3]{CZ1-08} gives
$$
\deg \gcd(U-1,V-1)\ll \max(\deg U,\deg V)^{2/3}.
$$
However,~\cite[Corollary 2.3]{CZ1-08} applies to more general situations.

It is interesting to  unify Theorems~\ref{thm:genAR1} 
and~\ref{thm:genAR2} and obtain a similar result  for 
$$
\gcd\(h_1\(f_1^{n_1}\cdots f_{\ell}^{n_{\ell}}\),h_2\(g_1^{m_1}\cdots g_r^{m_r}\)\),
$$
where $h_1,h_2\in\C[T]$. Similar ideas may work for this case however 
they require a uniform bound for the number of points on intersections of curves in 
$\G_{m}^{\ell+r}$ with algebraic subgroups of dimension $k \le \ell+r-2$ in 
Lemma~\ref{lem:bmz}. We note that for $\ell=r=1$ this was possible due to the uniform bounds of~\cite{BS}. However, no such bounds are available in the more general case that we need. 

\subsection{Our results: multivariate case} 
For our first result in the multivariate case, we reduce the problem to the univariate case using Hilbert's irreducibility theorem (see Section~\ref{sec:Hilb}), and to control the degree for such specialisation we also couple this approach with a transformation involving the Kronecker substitution. We obtain:

\begin{theorem}
\label{thm:multivAR}
Let $h_1,h_2\in\C[T]$  and $F,G\in\C[X_1,\ldots,X_{\ell}]$. We denote by $D=\max_{i=1\ldots,\ell}\(\deg_{X_i}F,\deg_{X_i}G\)$.  If $F, G$ are multiplicatively independent in $\C(X_1,\ldots,X_{\ell})^*/\C^*$, then for all $n,m\ge 1$ we have
 $$
\deg\gcd\(h_1\(F^n\),h_2\(G^m\)\)\le  d_{h_1} d_{h_2} \(44(D+1)^{2{\ell}}\)^{(D+1)^{\ell}}.
$$
\end{theorem}

We note that if $h_1=h_2=T-1$ as in~\cite[Theorem 1]{AR}, then in Theorem~\ref{thm:multivAR} we need $F,G$ to be just multiplicatively independent.

Theorem~\ref{thm:multivAR} is proved in Section~\ref{sec:multivar1}.

 Another natural extension of~\cite[Theorem~1]{AR} to the multivariate case is related to the fact that the greatest common divisor of two univariate polynomials is given by their common zeros. 
Thus~\cite[Theorem~1]{AR} says that the number of common zeros of $f^n-1$ and $g^m-1$, for two polynomials $f,g\in\C[T]$, is bounded by a constant depending only on $f$ and $g$ for all  $n,m\ge 1$, and Lemma~\ref{lem:univAR} gives a uniform bound.

For positive integers $\ell,D\ge 1$, we denote 
 \begin{equation}
 \label{eq:gamma}
 \gamma_{\ell}(D)=\binom{\ell+1+D^\ell}{\ell+1}.
 \end{equation}

We now obtain the following result proved in Section~\ref{sec:multivar2}. This multivariate generalisation is based on a result of Granville and Rudnick~\cite[Corollary 3.1]{GranRud}, which describes the structure of torsion points on hypersurfaces, see Lemma~\ref{lem:GranRud}.

\begin{theorem}
\label{thm:CommonZeros}
Let $F_1,\ldots,F_{\ell+1}\in\C[X_1,\ldots,X_{\ell}]$ be multiplicatively independent polynomials of degree at most $D$. Then,  
$$\bigcup_{n_1,\ldots,n_{\ell+1}\in\N}Z\(F_1^{n_1}-1,\ldots,F_{\ell+1}^{n_{\ell+1}}-1\)
$$ 
is contained in at most 
$$
        N \le  (0.792\gamma_{\ell}(D)/\log \(\gamma_{\ell}(D)+1\))^{\gamma_{\ell}(D)}
$$
algebraic varieties, 
each defined by at most $\ell+1$ polynomials of
 degree at most $(\ell+1)D^{\ell}\prod_{p\le  \gamma_{\ell}(D)}p$, where the product runs over all primes $p\le  \gamma_{\ell}(D)$. 
\end{theorem}

We recall that by the prime number theorem, for an integer $k\ge 1$,
$$
\prod_{p\le k}p = \exp(k + o(k)).
$$

We note that the bound on the number of algebraic subgroups that contain the points on $\cH$ with coordinates roots of unity may also follow from~\cite[Theorem 1.1 and Lemma 2.6]{AS}, which says that, for a hypersurface defined by $H\in\C[X_1,\ldots,X_s]$, $s\ge 2$, of degree $D$,  the number of maximal torsion cosets contained in $\cH$ is at most
$$
c_1(s)D^{c_2(s)}
$$
with
\begin{equation*}
\label{eq:c1c2}
c_1(s)=s^{\frac{3}{2}(2+s)5^{s}}\quad \textrm{and}\quad c_2(s)=\frac{1}{16}\(49 \cdot 5^{s-2}-4s-9\).
\end{equation*}

We also note that any argument that is based on the Bezout Theorem  ultimately leads to  bounds that depend on the exponents $n_1,\ldots,n_{\ell+1}$, while the bounds of Theorem~\ref{thm:CommonZeros}
depend only on the initial data.

We conclude the paper with comments on future work.

\section{Preliminaries}
\label{sec:prel}

\subsection{The Ailon-Rudnick Theorem}
\label{sec:ART}

The Ailon-Rudnick theorem is based on a well-known conjecture of Lang, proved by Ihara, Serre and Tate~\cite{L}, which says that a plane curve, which does not contain a translate of an algebraic subgroup of $\G_m^2$, contains only finitely many torsion points. In this case, Beukers and Smyth~\cite[Section~4.1]{BS} give a uniform bound for the number of such points (see Lemma~\ref{lem:BS} below), and Corvaja and Zannier~\cite{CZ08} give an upper bound (actually for curves in $\G_m^n$) for the maximal order of torsion points  on the curve.

We now present the result of Ailon and Rudnick~\cite[Theorem 1]{AR}, coupled with the result of Beukers and Smyth~\cite[Section~4.1]{BS}, which we first mention separately.

\begin{lemma}
\label{lem:BS}
An algebraic curve $H(X,Y)=0$ has at most $11(\deg H)^2$ points which are roots of unity, unless $H$ has a factor of the form $X^i-\rho Y^j$ or $X^iY^j-\rho$
for some nonnegative integers $i,j$ not both zero and some root of unity $\rho$.
\end{lemma}

Using Lemma~\ref{lem:BS}, we obtain the following more precise form of~\cite[Theorem 1]{AR}. 

\begin{lemma}
\label{lem:univAR}
Let $f,g\in\C[T]$ be non constant polynomials. If $f$ and $g$ are multiplicatively independent, then for all $n,m\ge 1$, we have
$$
\deg\gcd\(f^n-1,g^m-1\)\le \(11(d_f+d_g)^2\)^{\min(d_f,d_g)}.
$$
\end{lemma}

\begin{proof}
The proof, except the explicit bound for the degree, is given in~\cite[Theorem~1]{AR}. To see the degree bound, we just apply Lemma~\ref{lem:BS}.

Our curve is given in parametric form $\{\(f(t),g(t)\)~:~ t\in\C\}$ and we need to find the degree of the implicit form $H$ such that $H(f(t),g(t))=0$, $t\in\C$. This is obtained using resultants, that is 
$$
H=\Res_T\(f(T)-X_1,g(T)-X_2\),
$$ 
which is a polynomial of degree $\deg g$ in $X_1$ and $\deg f$ in $X_2$. Thus, the total degree of $H$ is 
at most $\deg f+\deg g$.

Let $\widetilde{H}$ be an absolutely irreducible factor of $H$ and assume that $\widetilde{H}(f(t),g(t))=0$ for infinitely many $t\in\C$. As $\widetilde{H}(f(T),g(T))$ is a  univariate polynomial , we must have the identity $\widetilde{H}(f(T)),g(T))=0$. Then, by Lemma~\ref{lem:BS} applied with the curve defined by the polynomial $\widetilde{H}$, we obtain that $\widetilde{H}$ is of the form $X_1^{n_1}X_2^{n_2}=\omega$, for some root of unity $\omega$ and integers $n_1,n_2$ not both zero.  This implies that $f,g$ are multiplicatively dependent, which contradicts the hypothesis. Thus, there is no such absolutely irreducible divisor of $H$.

Therefore $\gcd\(f^n-1,g^m-1\)$ 
has at most $11(d_f+d_g)^2$ distinct zeros. 
As in the proof of~\cite[Theorem~1]{AR}, we see that the multiplicity
of each zero is at most $\min(d_f,d_g)$.
 The bound now follows. 
\end{proof}

\subsection{Intersection of curves with algebraic groups}
\label{sec:intesect}

We define $\G_m^k$ as the set of $k$-tuples of non-zero complex numbers equiped with the group law defined by component-wise multiplication. We refer to~\cite[Appendix by Umberto Zannier]{Sch} for necessary definitions on algebraic subgroups.

One of the main tools in our paper is a result on the finiteness of the number of points on the intersection of a curve in $\G_m^k$ with algebraic subgroups of codimension at least $2$, initially obtained in~\cite{BMZ99} for curves over $\ovQ$, and later on extended over $\C$, see~\cite{BMZ08,Mau} and references therein. We present it in the following form.

\begin{lemma}
\label{lem:bmz}
Let $C\subset \G_m^{k}$, $k\ge 2$, be an irreducible curve over $\C$. Assume that for every nonzero vector $(r_1,\ldots,r_{k})\in\Z^k$ the monomial $X_1^{r_1}\cdots X_{k}^{r_{k}}$ is not identically $1$ on $C$. Then there are finitely many points $(x_1,\ldots,x_{k})\in C(\C)$ for which there exist linearly independent vectors $(a_1,\ldots,a_{k}), (b_1,\ldots,b_{k})$ in $\Z^{k}$ such that
$$
x_1^{a_1}\cdots x_{k}^{a_{k}}=x_1^{b_1}\cdots x_{k}^{b_{k}}=1.
$$
\end{lemma}

\begin{remark}
\label{rem:translate}
As explained in~\cite{BMZ99}, the condition of Lemma~\ref{lem:bmz} that the monomial $X_1^{r_1}\cdots X_{k}^{r_{k}}$ is not identically $1$ on $C$ is equivalent with the curve not being contained in a proper subtorus of $\G_m^{k}$.
\end{remark}

\subsection{Torsion points on hypersurfaces} Results regarding uniform bounds on the number of torsion points in subvarieties of $\G_m^k$ go back to work of Bombieri and Zannier~\cite{BZ95}, Schlickewei~\cite{Schl} and Evertse~\cite{Ever}. For example, Evertse~\cite{Ever}, improving bounds of Schlickewei~\cite{Schl}, shows that the number of non-degenerate solutions  in roots of unity to the equation $a_1x_1+\cdots+a_kx_k=1$, $a_1,\ldots,a_k\in\C$, is at most $(k+1)^{3(k+1)^2}$. 

For our results we use the following result of Granville and Rudnick, see~\cite[Corollary 3.1]{GranRud}, which describes the structure of the algebraic subgroups that contain the roots of unity on a hypersurface. 
Although the statement of their result does not contain the bound for the degree or the number of the polynomials defining the algebraic subgroups, this follows directly from or  is explicitly stated in their proof. Moreover, we recall this result only for the case of hypersurfaces, however their result holds for any algebraic variety.

\begin{lemma}
\label{lem:GranRud}
Let $\cH=Z(H)$ be a hypersurface in $\C^k$ defined by a polynomial $H\in\C[X_1,\ldots,X_k]$ of degree $D$ and with $s(H)$ terms. There exists a finite list $\cB$ of at most 
$$
N(H) \le  (0.792s(H)/\log \(s(H)+1\))^{s(H)}
$$
integer
$k\times k$ matrices $B=(b_{j,i})$, $i,j=1,\ldots,k$, 
and $$b_{j,i}\le D\prod_{p\le s(H)}p,$$ where the product runs over all primes $p\le s(H)$, such that if $\xi\in\cH$ is a torsion point, then $\xi\in\cup_{B\in\cB}W_B$, where $$W_B=\bigcap_{j=1}^{k}Z\(X_1^{b_{j,1}}X_2^{b_{j,2}}\cdots X_k^{b_{j,k}}-1\).$$
\end{lemma}

\begin{proof}
The proof is 
         essentially 
given in~\cite[Corollary 3.1]{GranRud}. Indeed, each matrix $B$ corresponds to a partition of the set $\{1,2,\ldots,s(H)\}$, and thus, the number of matrices $B$ in the set $\cB$ is given by  the number of such partitions, which, by~\cite[Theorem 2.1]{BT}, is at most $N(H)$.

        The number of rows $n_B$ of a matrix $B \in \cB$ is not 
        specified in~\cite[Corollary 3.1]{GranRud}. 
        However, we can choose the largest linear independent set of 
        these vectors  $\vec{b}_j$, which is of
        cardinality at most $k$, and all other varieties of the 
        form  $Z\(X_1^{b_{i,1}}X_2^{b_{i,2}}\cdots 
        X_k^{b_{i,k}}-1\)$ are defined by combinations of these vectors. 
        Thus, we can consider $n_B\le k$.
        Repeating some rows if necessary we can take $n_B=k$ which 
        concludes the proof. \end{proof}

\subsection{Hilbertian fields and multiplicative independence} 
\label{sec:Hilb}

For the first multivariate generalisation of~\cite[Theorem 1]{AR} we need  a result which says that given  $F_1,\ldots,F_s\in\C[X_1,\ldots,X_{\ell}]$ that are multiplicatively independent in  $\C(X_1,\ldots,X_{\ell})^*/\C^*$, there exists a specialisation $(\alpha_2,\ldots,\alpha_{\ell})\in\C^{\ell-1}$ such that $F_i(X_1,\alpha_2,\ldots,\alpha_{\ell})$, $i=1,\ldots,s$, are multiplicatively independent in $\C(X_1)^*/\C^*$ (see Lemma~\ref{lem:multindep} below). Such a result follows directly from~\cite[Theorem 1]{BMZ99} which says that the points lying in the intersection of a curve $\cC$, not contained  in any translate of a proper subtorus of $\G_m^{\ell}$, with the union of all proper algebraic subgroups is of bounded height (see also~\cite[Theorem 1']{BMZ99}).

Furthermore, this 
also follows from previous work of N\' eron~\cite{Ne} (see also~\cite[Chapter 11]{Serre}), Silverman~\cite[Theorem C]{Si83} and Masser~\cite{Ma89} (see also~\cite[Notes to Chapter 1]{Za} where Masser's method is explained) on specialisations of finitely generated subgroups of abelian varieties. 
In particular, Masser's result~\cite{Ma89} gives explicit bounds for the least degree of a hypersurface containing the set of exceptional points, that is, points that lead to multiplicative dependence, of bounded degree and height.

Although the above results are sufficient for our purpose, for the sake of completeness 
we now  give a simple self-contained proof 
that follows directly from Hilbert's irreducibility theorem, see~\cite[Theorem 46]{Sch}. 
Moreover, this proof does not appeal to the notion 
of height and applies to arbitrary  Hilbertian field (see Definition~\ref{def:HIT} below), rather than to just finite extensions of $\Q$.

\begin{definition}
\label{def:HIT}
We say that a field $\K$ is Hilbertian  
if for any irreducible polynomials $P_1,\ldots,P_r\in\K[X_1,\ldots,X_{\ell}]$ over $\K$  there exists a specialisation 
$(\alpha_2,\ldots,\alpha_{\ell})\in\K^{{\ell}-1}$ such that $P_i(X_1,\alpha_2,\ldots,\alpha_{\ell})$, $i=1,\ldots,r$, are all irreducible over $\K$.
\end{definition}

In particular, by the famous  Hilbert's irreducibility theorem, any finite extension of $\Q$ is a  Hilbertian field. Furthermore, by~\cite[Theorem 49]{Sch} every finitely generated infinite field and every finitely generated transcendental extension of an arbitrary field are Hilbertian. 

We also need the following simple fact.

\begin{lemma}
\label{lem:multconst} Let $\K$ be an arbitrary field. 
The polynomials $F_1,\ldots,F_{s}\in\K[X_1,\ldots,X_{\ell}]$ are multiplicatively independent in $\K(X_1,\ldots,X_{\ell})^*/\K^*$ if and only if 
$$
\frac{F_1}{F_1^*(\vec{0})},\ldots,\frac{F_{s}}{F_s^*(\vec{0})}
$$ 
are multiplicatively independent, where $F_i^*(\vec{0}) =1$ if 
$F_i(\vec{0})=0$ and $F_i^*(\vec{0}) =F_i(\vec{0})$ otherwise. 
\end{lemma}
\begin{proof}
Assume that  $F_1,\ldots,F_{s}\in\K[X_1,\ldots,X_{\ell}]$ are multiplicatively independent in $\K(X_1,\ldots,X_{\ell})^*/\K^*$, but there exist integers $i_1,\ldots,i_{\ell}$ not all zero, such that
$$
\(\frac{F_1}{F_1^*(\vec{0})}\)^{i_1}\cdots\(\frac{F_{s}}{F_s^*(\vec{0})}\)^{i_s}=1.
$$
Then, $F_1^{i_1}\cdots F_s^{i_s}=a$, where $a=\(F_1^*(\vec{0})\)^{i_1}\cdots\(F_s^*(\vec{0})\)^{i_s}$, and thus we obtain a contradiction. 

For the other implication, assume $F_1^{i_1}\cdots F_s^{i_s}=a$ for some $a\in\K^*$ and integers $i_1,\ldots,i_s$ not all zero. From here we get again that 
$
a=\(F_1^*(\vec{0})\)^{i_1}\cdots\(F_s^*(\vec{0})\)^{i_s}$ (we note that if $F_k(\vec{0})=$ for some $k$, then we need to have $i_k=0$ as otherwise we get $a=0$, which is a contradiction). 
Thus we obtain again a contradiction with the fact that 
the polynomials
$F_1/F_1^*(\vec{0}),\ldots, F_{s}/F_s^*(\vec{0})$
are multiplicatively independent.
\end{proof}

We note that the conclusion of Lemma~\ref{lem:multconst} holds with any $\bfalpha\in\K^{\ell}$, that is, $F_1,\ldots,F_{s}\in\K[X_1,\ldots,X_{\ell}]$ are multiplicatively independent in $\K(X_1,\ldots,X_{\ell})^*/\K^*$ if and only if $\frac{F_1}{F_1^*(\bfalpha)},\ldots,\frac{F_{s}}{F_s^*(\bfalpha)}$ are multiplicatively independent for any $\bfalpha\in\K^{\ell}$, where as before $F_i^*(\bfalpha) =1$ if 
$F_i(\bfalpha)=0$ and $F_i^*(\bfalpha) =F_i(\bfalpha)$ otherwise. 

The following result is easily derived from Lemma~\ref{lem:multconst}.

\begin{lemma}
\label{lem:multindep}
Let $\K$ be a Hilbertian field  and $F_1,\ldots,F_s\in\K[X_1,\ldots,X_{\ell}]$ multiplicatively independent polynomials in $\K(X_1,\ldots,X_{\ell})^*/\K^*$.
Then,  there exists a specialisation $(\alpha_2,\ldots,\alpha_{\ell})\in\K^{{\ell}-1}$ such that the polynomials $F_i(X_1,\alpha_2,\ldots,\alpha_{\ell})$, $i=1,\ldots,s$, are multiplicatively independent in $\K(X_1)^*/\K^*$.
\end{lemma}
\begin{proof}
By Lemma~\ref{lem:multconst} the polynomials $F_1,\ldots,F_s$ are multiplicatively independent in $\K(X_1,\ldots,X_{\ell})^*/\K^*$ if and only if $\frac{F_1}{F_1^*(\vec{0})},\ldots,\frac{F_{s}}{F_s^*(\vec{0})}$ are multiplicatively independent.

We denote $G_i=\frac{F_i}{F_i^*(\vec{0})}$, $i=1,\ldots,s$. Let $P_1,\ldots,P_r\in\K[X_1,\ldots,X_{\ell}]$ be the distinct irreducible factors of $G_1,\ldots,G_s$, that is, we have the factorisation
$$
G_i=P_1^{e_{i,1}}\cdots P_r^{e_{i,r}},\quad i=1,\ldots,s.
$$

We note that the polynomials $G_1,\ldots,G_s$ being multiplicatively independent is equivalent with the matrix $(e_{i,j})_{\substack{1\le i\le s\\1\le j\le r}}$ having full rank.

Since  $\K$ is Hilbertian,  there exists a specialisation $(\alpha_2,\ldots,\alpha_{\ell})\in\K^{{\ell}-1}$ such that $P_j(X_1,\alpha_2,\ldots,\alpha_{\ell})$, $j=1,\ldots,r$, are all irreducible over $\K$ and for $i=1,\ldots,s$, we have the factorisation 
$$
G_i(X_1,\alpha_2,\ldots,\alpha_{\ell})=P_1(X_1,\alpha_2,\ldots,\alpha_{\ell})^{e_{i,1}}\cdots P_r(X_1,\alpha_2,\ldots,\alpha_{\ell})^{e_{i,r}}.
$$
If the polynomials $G_i(X_1,\alpha_2,\ldots,\alpha_{\ell})$, $i=1,\ldots,s$, are multiplicative dependent over $\K$, then there exist integers $\ell_1,\ldots,\ell_s$, not all zero such that
$$
G_1(X_1,\alpha_2,\ldots,\alpha_{\ell})^{\ell_1}\cdots G_s(X_1,\alpha_2,\ldots,\alpha_{\ell})^{\ell_s}=1.
$$
This is equivalent to the fact that the matrix $(e_{i,j})_{\substack{1\le i\le r\\1\le j\le s}}$ does not have full rank, which contradicts the fact that the initial polynomials $G_1,\ldots,G_s\in\K[X_1,\ldots,X_{\ell}]$ are multiplicatively independent. 

Recalling the definition of the polynomials $G_i$, $i=1,\ldots,s$, 
and applying again Lemma~\ref{lem:multconst}, we get that the polynomials $F_i(X_1,\alpha_2,\ldots,\alpha_{\ell})$, $i=1,\ldots,s$, are multiplicatively independent in $\K(X_1)^*/\K^*$, which concludes the proof.
\end{proof}

\subsection{Multiplicities of zeroes}

To prove Theorem~\ref{thm:genAR2} we need a uniform bound for the multiplicities of zeros of polynomials of the form $f_1^{n_1}\cdots f_{\ell}^{n_{\ell}}-g_1^{m_1}\cdots g_r^{m_r}$. We present such a result below, as well as deduce as a consequence a similar uniform bound for rational functions, which we hope to be of independent interest.

For a rational function  $h\in\C(T)$, we denote by $\Mult(h)$ the largest multiplicity and by $Z(h)$ the set of zeros of $h$ in $\C$, 
respectively. We also recall that for a polynomial $f\in\C[T]$, we use the notation $d_f$ for the degree of $f$.

\begin{lemma}
\label{lem:ABCpoly}
Let $f_1,\ldots,f_{\ell},g_1,\ldots,g_r\in\C[T]$ be polynomials satisfying $Z(f_1\cdots f_{\ell})\cap Z(g_1\cdots g_{\ell})=\emptyset$. Then, for all $n_1,\ldots,n_{\ell},m_1,\ldots,m_r\ge 0$, we have
$$
\Mult\(f_1^{n_1}\cdots f_{\ell}^{n_{\ell}}-g_1^{m_1}\cdots g_r^{m_r}\)\le \sum_{i=1}^{\ell}d_{f_i}+\sum_{j=1}^rd_{g_j}.
$$
\end{lemma}
\begin{proof} We denote $\vec{n}=(n_1,\ldots,n_{\ell})\in\N^{\ell}$ and $\vec{m}=(m_1,\ldots,m_r)\in\N^r$. 

Writing the factorisation into linear factors, we have
$$
f_1^{n_1}\cdots f_{\ell}^{n_{\ell}}-g_1^{m_1}\cdots g_{r}^{m_r}=a_{\vec{n},\vec{m}}\prod_{t\in Z\(f_1^{n_1}\cdots f_{\ell}^{n_{\ell}}-g_1^{m_1}\cdots g_{r}^{m_r}\)}(T-t)^{e_t},
$$
where $a_{\vec{n},\vec{m}}\in\C$ is the leading coefficient of $f_1^{n_1}\cdots f_{\ell}^{n_{\ell}}-g_1^{m_1}\cdots g_{r}^{m_r}$.

For simplicity we denote by $$S_{\vec{n},\vec{m}}=Z\(f_1^{n_1}\cdots f_{\ell}^{n_{\ell}}-g_1^{m_1}\cdots g_{r}^{m_r}\).$$

Let  
$\Mult=\max_{t\in S_{\vec{n},\vec{m}}} e_t$ be the largest  multiplicity of the zeros of $f_1^{n_1}\cdots f_{\ell}^{n_{\ell}}-g_1^{m_1}\cdots g_{r}^{m_r}$.

The bound for $\Mult$  follows immediately from the polynomial $ABC$ theorem (proved first by Stothers~\cite{St}, and then independently by Mason~\cite{Mas} and Silverman~\cite{Si84}). Indeed, we apply the polynomial $ABC$ theorem with $A=a_{\vec{n},\vec{m}}\prod_{t\in S_{\vec{n},\vec{m}}}(T-t)^{e_t}$, $B=f_1^{n_1}\cdots f_{\ell}^{n_{\ell}}$ and $C=g_1^{m_1}\cdots g_{r}^{m_r}$, which are pairwise corpime. We get
\begin{equation}
\label{eq:l}
\sum_{t\in S_{\vec{n},\vec{m}}}e_t\le \sum_{i=1}^{\ell}d_{f_i}+\sum_{j=1}^rd_{g_j}+\#S_{\vec{n},\vec{m}}-1.
\end{equation}
Taking into account that $$\sum_{t\in S_{\vec{n},\vec{m}}}e_t\ge \Mult +\#S_{\vec{n},\vec{m}}-1,$$ from~\eqref{eq:l} we obtain
$$
\Mult\le \sum_{i=1}^{\ell}d_{f_i}+\sum_{j=1}^rd_{g_j},
$$
which concludes the proof. 
\end{proof}

We present now a similar result for rational functions.

\begin{corollary}
\label{cor:ABCrat}
Let $h_1,\ldots,h_{\ell}\in\C(T)$, $h_i=f_i/g_i$, $f_i,g_i\in\C[T]$, $i=1,\ldots,\ell$, with $Z(f_1\cdots f_{\ell})\cap Z(g_1\cdots g_{\ell})=\emptyset$. Then, for all $n_1,\ldots,n_{\ell}\ge 0$, we have 
$$
\Mult\(h_1^{n_1}\cdots h_{\ell}^{n_{\ell}}-1\)\le \sum_{i=1}^{\ell} \(\deg f_i+\deg g_i\).
$$
\end{corollary}

\begin{proof}
We note that $Z\(h_1^{n_1}\cdots h_{\ell}^{n_{\ell}}-1\)=Z\(f_1^{n_1}\cdots f_{\ell}^{n_{\ell}}-g_1^{n_1}\cdots g_{\ell}^{n_{\ell}}\)$. The result now follows directly from Lemma~\ref{lem:ABCpoly} applied with $r=\ell$ and $m_i=n_i$, $i=1,\ldots,\ell$.
\end{proof}

\subsection{Algebraic dependence} 
We need the following result~\cite[Theorem~1.1]{Plo} which gives a 
degree bound for the annihilating polynomial of 
          algebraically dependent polynomials,   which is always the case when     
          the number of polynomials exceeds the number of variables.
The result holds over any field, but we present it only over $\C$.

\begin{lemma}
\label{lem:perron}
Let $F_1,\ldots,F_{\ell+1}\in\C[X_1,\ldots,X_{\ell}]$ be of degree at most $D$. Then there exists a nonzero polynomial $R\in\C[Z_1,\ldots,Z_{\ell+1}]$ of degree at most $D^{\ell}$ such that $R(F_1,\ldots,F_{\ell+1})=0$.
\end{lemma}

\section{Proofs of Main Results}
\label{sec:main}

\subsection{Proof of Theorem~\ref{thm:genAR1}}
\label{sec:univar1}
We use the same idea as in the proof of~\cite[Theorem 1]{AR}  and Lemma~\ref{lem:univAR}. Indeed, we write the factorisation in linear factors,
$$
h_1=\prod_{i=1}^{d_{h_1}}(T-\omega_{1,i}),\quad h_2=\prod_{i=1}^{d_{h_2}}(T-\omega_{2,i}),
$$
where $\omega_{1,i},\omega_{2,j}\in\C$, $i=1,\ldots,d_{h_1}$, $j=1,\ldots,d_{h_2}$.

Thus, we reduce the problem to estimating the degree of each
$$
\gcd \(f^n-\omega_{1,i},g^m-\omega_{2,j}\).
$$
For simplicity we use the notation $\omega_1$ and $\omega_2$ for any two roots of $h_1$ and $h_2$, respectively, and we denote
$$
\cD_{n,m}(\omega_1,\omega_2)=\gcd \(f^n-\omega_{1},g^m-\omega_{2}\).
$$

Thus, as in Lemma~\ref{lem:univAR}, we need to bound the number of $t\in\C$ such that $f(t)^n=\omega_1$ and $g(t)^m=\omega_2$ for some positive integers $n$ and $m$. 

For every $n,m\ge 1$, we fix an element $t_{n,m}\in\C$ such that 
\begin{equation}
\label{eq:root}
f(t_{n,m})^n=\omega_1,\quad g(t_{n,m})^m=\omega_2
\end{equation} 
(if no such $t_{n,m}$ exists then we immediately have 
$\deg \cD_{n,m}(\omega_1,\omega_2) =0$).
We define new polynomials
$$
\widetilde{f}_{n,m}(T)=\frac{1}{f(t_{n,m})}f(T)\mand
 \widetilde{g}_{n,m}(T)=\frac{1}{g(t_{n,m})}g(T).
$$
As $f$ and $g$ are multiplicatively independent in $\C(T)^*/\C^*$, we obtain that $\widetilde{f}_{n,m}$ and $\widetilde{g}_{n,m}$ are multiplicatively independent for every $n,m$.

Thus, we can apply Lemma~\ref{lem:univAR} and conclude that
$$
\deg \gcd\(\widetilde{f}_{n,m}^n-1,\widetilde{g}_{n,m}^m-1\)\le \(11(d_f+d_g)^{2}\)^{\min(d_f,d_g)}.
$$

From~\eqref{eq:root} and the definition of $\widetilde{f}_{n,m}$ and $\widetilde{g}_{n,m}$, we have
$$
\deg \cD_{n,m}(\omega_1,\omega_2)=\deg \gcd\(\widetilde{f}_{n,m}^n-1,\widetilde{g}_{n,m}^m-1\),
$$
and thus, for every $n,m\ge 1$, we get
$$
\deg\cD_{n,m}(\omega_1,\omega_2)\le \(11(d_f+d_g)^{2}\)^{\min(d_f,d_g)}.
$$
As this holds for any roots $\omega_1,\omega_2$ of $h_1$ and $h_2$, respectively, we get
$$
\deg \gcd\(h_1\(f^n\),h_2\(g^m\)\) 
\le d_{h_1}d_{h_2} \(11(d_f+d_g)^{2}\)^{\min(d_f,d_g)}
$$
which concludes the proof.
\qed

\subsection{Proof of Theorem~\ref{thm:genAR2}}
\label{sec:univar2}

We use the same idea as in the proof of~\cite[Theorem1]{AR} combined with 
Lemma~\ref{lem:bmz}.

First, we note that for any zero $t\in \C$ of 
$$ 
\gcd\(\prod_{i=1}^{\ell}f_i^{n_i}-\prod_{i=1}^k\varphi_i^{\nu_i},\prod_{i=1}^rg_i^{m_i}-\prod_{i=1}^s\psi_i^{\mu_i}\)
$$ 
the condition~\eqref{eq:cond} ensures that $\varphi_i(t),\psi_j(t)\ne0$, $i=1,\ldots,l$, $j=1,\ldots,k$.
Therefore each such zero $t$ 
satisfies
\begin{equation}
\label{eq:ffgg}
\prod_{i=1}^{\ell}f_i(t)^{n_i}\cdot\prod_{i=1}^k\varphi_i(t)^{-\nu_i}=\prod_{i=1}^rg_i(t)^{m_i}\cdot\prod_{i=1}^s\psi_i(t)^{-\mu_i}=1.
\end{equation}

We apply Lemma~\ref{lem:bmz} with $k$ replaced by $L={\ell}+k+r+s$ and with the curve 
\begin{equation*}
\begin{split}
C= \bigl\{(f_1(t),\ldots,f_{\ell}(t),\varphi_1(t)&,\ldots,\varphi_{k}(t),g_1(t),\ldots,g_r(t),\\
&\psi_1(t),\ldots,\psi_s(t))~:~  t\in\C\bigr\} \subseteq\G_m^{L}.
\end{split}
\end{equation*}
Indeed, we denote $$\vec{v}=(n_1,\ldots,n_{\ell},-\nu_1,\ldots,-\nu_k),\ \vec{w}=(m_1,\ldots,m_r,-\mu_1,\ldots,-\mu_s).$$ As the vectors
$$
(\vec{v},\vec{0}),\ (\vec{0},\vec{w})\in \Z^L
$$
are linearly independent, by Lemma~\ref{lem:bmz} we obtain that there are finitely many $t\in\C$ such that~\eqref{eq:ffgg} holds 
for some vectors $\vec{v},\vec{w}$ as above.

We denote by $S$ the set of such $t\in\C$. For $\vec{v},\vec{w}$, we denote 
$$
\cD_{\vec{v},\vec{w}}=\gcd\(\prod_{i=1}^{\ell}f_i^{n_i}-\prod_{i=1}^k\varphi_i^{\nu_i},\prod_{i=1}^rg_i^{m_i}-\prod_{i=1}^s\psi_i^{\mu_i}\).
$$
We see from the above that set  of zeros $Z(\cD_{\vec{v},\vec{w}})$ belongs to some fixed set that depends only 
on the above curve $C$ and thus only on the polynomials 
in the initial data.  
To give the upper bound for $\deg \cD_{\vec{v},\vec{w}}$ we only need to prove that the multiplicity of the roots $t\in S$ of $\cD_{\vec{v},\vec{w}}$ 
   can be bounded uniformly for all  
integer vectors $\vec{v},\vec{w}$. 
This is given by Lemma~\ref{lem:ABCpoly} 
applied with  the polynomials 
$\prod_{i=1}^{\ell}f_i^{n_i}-\prod_{i=1}^k\varphi_i^{\nu_i}$ and $\prod_{i=1}^rg_i^{m_i}-\prod_{i=1}^s\psi_i^{\mu_i}$. 

Indeed, if we denote by $\Mult_1$ and $\Mult_2$ the largest multiplicity of roots in $S$ of the first and second polynomials, respectively,  we get
$$
\Mult_1\le \sum_{i=1}^{\ell} d_{f_i}+\sum_{i=1}^kd_{\varphi_i},\quad \Mult_2\le  \sum_{i=1}^{r} d_{g_i}+\sum_{i=1}^sd_{\psi_i}.
$$

Thus, there exists a polynomial $h\in\C[T]$ defined by 
$$
h=\prod_{t\in S} (T-t)^{d},\quad d=\min\(\sum_{i=1}^{\ell} d_{f_i}+\sum_{i=1}^kd_{\varphi_i},\sum_{i=1}^{r} d_{g_i}+\sum_{i=1}^sd_{\psi_i}\)
$$
such that $\cD_{\vec{v},\vec{w}}\mid h$ for every vectors $\vec{v},\vec{w}$ as above.
      This concludes the proof of Part~{\bf{i}}.

   For Part~{\bf{ii}},
for each $t\in S$, let 
$$\cL_t=\{\(\vec{v},\vec{w}\)\in \N^L
~:~(T-t)\mid \cD_{\vec{v},\vec{w}}\}.$$
We note that $\cL_s$ is actually a monoid as the sum of any two elements in $\cL_t$ is also an element of $\cL_t$. As the set $S$ is finite, there are finitely many such monoids $\cL_t$, $t\in S$, such that $\deg \cD_{\vec{v},\vec{w}}\ge 1$ for any $\(\vec{v},\vec{w}\)\in\cL_t$. 

We are left to show that $\cup_{t\in S}\cL_t$ is not the entire space $\N^{L}$. Indeed this follows directly from~\cite[Theorem~1]{AR} as for the diagonal case, that is $\vec{v}=n(\underbrace{1,\ldots,1}_{\ell},0,\ldots,0)\in\N^{\ell+k}$ and $\vec{w}=n(\underbrace{1,\ldots,1}_{r},0,\ldots,0)\in\N^{r+s}$, we have
$$\gcd\(\(f_1\cdots f_{\ell}\)^n-1,\(g_1\cdots g_r\)^n-1\)=1$$
infinitely often. 

Thus, for any $\(\vec{v},\vec{w}\)$ outside $\cup_{t\in S}\cL_t$, we have $\cD_{\vec{v},\vec{w}}=1$, and  we conclude the proof.
\qed

\subsection{Proof of Theorem~\ref{thm:multivAR}}
\label{sec:multivar1}
The idea of the proof lies in applying Hilbert's irreducibility theorem, and in particular Lemma~\ref{lem:multindep}, to reduce via specialisations to the univariate case and thus use Theorem~\ref{thm:genAR1}.

We denote $d=D+1$, that is $d>\deg_{X_j}F,\deg_{X_j}G$ for any $j=1,\ldots,{\ell}$. We define the polynomials 
\begin{equation*}
\begin{split}
&\widetilde{F}(X_1,\ldots,X_{\ell})=F\(X_1,X_2+X_1^d,\ldots,X_{\ell}+X_1^{d^{{\ell}-1}}\),\\
&\widetilde{G}(X_1,\ldots,X_{\ell})=G\(X_1,X_2+X_1^d,\ldots,X_{\ell}+X_1^{d^{{\ell}-1}}\).
\end{split}
\end{equation*}

The polynomials $\widetilde{F},\widetilde{G}$ have the property that 
$$
\deg \widetilde{F},\deg \widetilde{G}\le D\frac{d^{{\ell}}-1}{d-1}<(D+1)^{\ell}
$$ 
and 
$$
\deg \widetilde{F}(X_1,\alpha_2,\ldots,\alpha_{\ell})=\deg \widetilde{F},\ \deg \widetilde{G}(X_1,\alpha_2,\ldots,\alpha_{\ell})=\deg \widetilde{G}
$$
for any specialisation $(\alpha_2,\ldots,\alpha_{\ell})\in\C^{{\ell}-1}$.

Moreover, we note that the polynomials $\widetilde{F}, \widetilde{G}$ are also multiplicatively independent in $\C(X_1,\ldots,X_{\ell})^*/\C^*$. Indeed, if this would not be the case, then there exist $i_1,i_2$ not both zero and $a\in\C$ such that
$$
\widetilde{F}^{i_1}\widetilde{G}^{i_2}=a.
$$
Composing this polynomial identity with the polynomial automorphism 
\begin{equation}
\label{eq:Kro}
(X_1,\ldots,X_{\ell})\to \(X_1,X_2-X_1^d,\ldots,X_{\ell}-X_1^{d^{{\ell}-1}}\)
\end{equation}
we obtain that the polynomials $F,G$ are multiplicatively dependent in $\C(X_1,\ldots,X_{\ell})^*/\C^*$ and thus we get a contradiction.

Let $\K$ be the finite extension of $\Q$ by the coefficients of the polynomials $F,G$. 
By the Hilbert's irreducibility theorem, see~\cite[Theorem 46]{Sch},  $\K$ is a Hilbertian field. 
We apply now Lemma~\ref{lem:multindep} with the polynomials $\widetilde{F},\widetilde{G}$, and thus there exists a specialisation $(\alpha_2,\ldots,\alpha_{\ell})\in\K^{{\ell}-1}$ such that $\widetilde{F}(X_1,\alpha_2,\ldots,\alpha_{\ell})$ and $\widetilde{G}(X_1,\alpha_2,\ldots,\alpha_{\ell})$ are multiplicatively independent in $\C(X_1)^*/\C^*$. For simplicity, we denote $f=\widetilde{F}(X_1,\alpha_2,\ldots,\alpha_{\ell})$ and $g=\widetilde{G}(X_1,\alpha_2,\ldots,\alpha_{\ell})$.

We denote $\cD_{n,m}=\gcd\(h_1\(F^n\),h_2\(G^m\)\)$. Moreover,  we note that
$$
\cD_{n,m}\(X_1,X_2+X_1^d,\ldots,X_{\ell}+X_1^{d^{{\ell}-1}}\)=\gcd\(h_1\(\widetilde{F}^n\),h_2\(\widetilde{G}^m\)\).
$$
We denote $E_{n,m}=\gcd\(h_1\(\widetilde{F}^n\),h_2\(\widetilde{G}^m\)\)$, 
and for the specialisation $(\alpha_2,\ldots,\alpha_{\ell})$ one has
$$
E_{n,m}(X_1,\alpha_2,\ldots,\alpha_{\ell})\mid \gcd\(h_1\(f^n\),h_2\(g^m\)\).
$$
In particular, we have
$$
\deg \cD_{n,m}\le \deg E_{n,m}\le \deg \gcd\(h_1\(f^n\),h_2\(g^m\)\).
$$
We make here the remark that using the automorphism~\eqref{eq:Kro} was essential to have these degree inequalities, as if one just uses Hilbert's irreducibility theorem applied directly with the polynomials $F$ and $G$, we cannot guarantee that when we make specialisations we get that $\deg \cD_{n,m}\le \deg \gcd\(h_1\(f^n\),h_2\(g^m\)\)$.

We apply now Theorem~\ref{thm:genAR1} and using the fact that $\deg f,\deg g< (D+1)^{\ell}$  we conclude that
$$
 \deg \gcd\(h_1\(f^n\),h_2\(g^m\)\)\le d_{h_1} d_{h_2} (44(D+1)^{2\ell})^{(D+1)^{\ell}},
$$
which finishes the proof.
\qed

\subsection{Proof of Theorem~\ref{thm:CommonZeros}}
\label{sec:multivar2}
We define 
$$
\cH=\{(F_1(\bfalpha),\ldots,F_{\ell+1}(\bfalpha))\mid \bfalpha\in\C^{\ell}\}.
$$
    By Lemma~\ref{lem:perron} 
there exists a polynomial $R\in\C[Z_1,\ldots,Z_{\ell+1}]$ of degree at most $D^{\ell}$ such that $R(F_1,\ldots,F_{\ell+1})=0$. In other words, any point of $\cH$ is a point on the hypersurface defined by the zero set of $R$ in $\C^{{\ell+1}}$. 
In particular, any point $\bfalpha\in\C^{\ell}$ such that $F_i(\bfalpha)^{n_i}=1$, $i=1,\ldots,\ell+1$ gives a point on the hypersurface defined by the zero set of $R$ with coordinates roots of unity.

From Lemma~\ref{lem:GranRud} we get that there are at most
        $$
           N \le N(R)  \le  (0.792s(R)/\log \(s(R)+1\))^{s(R)}
       $$
algebraic subgroups, each defined by the zero set of at most $\ell+1$ polynomials of the form 
$$Z_1^{b_{j,1}}Z_2^{b_{j,2}}\cdots Z_{\ell+1}^{b_{j,\ell+1}}-1\in 
\C[Z_1, \ldots Z_{\ell+1}]
$$ 
with $$
\sum_{i=1}^{\ell+1}b_{j,i}\le (\ell+1)D^{\ell}\prod_{p\le s(R)}p,
\qquad j = 1, \ldots, \ell+1,
$$ 
where 
the product runs over all primes $p\le s(R)$,  that contain all the points in $Z(R)$ with coordinates roots of unity.  In particular, all points $\(F_1(\bfalpha),\ldots,F_{\ell+1}(\bfalpha)\)$ such that $F_i(\bfalpha)^{n_i}=1$, $i=1,\ldots,\ell+1$, lie in these algebraic subgroups.
        It remains to estimate $s(R)$.

As $R$ is a polynomial in $\ell+1$ variables and $\deg R\le D^{\ell}$, we have that $s(R)\le  \gamma_{\ell}(D)$, where $ \gamma_{\ell}(D)$ is defined by~\eqref{eq:gamma}.

Thus, the points $\bfalpha$ such that $F_i(\bfalpha)^{n_i}=1$, $i=1,\ldots,\ell+1$, lie in at most $N(H)$
algebraic varieties, each defined by at most $\ell+1$ polynomials of the form $F_1^{b_{j,1}}F_2^{b_{j,2}}\cdots F_{\ell+1}^{b_{j,\ell+1}}-1$ (note that these polynomials are non constant since $F_1,\ldots,F_{\ell+1}$ are multiplicatively independent) of degree at most 
$$
        \sum_{i=1}^{\ell+1}b_{j,i} \deg F_i
\le (\ell+1)D^{\ell+1}\prod_{p\le  \gamma_{\ell}(D)}p,\qquad j = 1, \ldots, \ell+1,
$$ where the product runs over all primes $p\le  \gamma_{\ell}(D)$. 
\qed

\section{Final Comments and Questions}

\subsection{Extensions over $\C$}
Lemma~\ref{lem:bmz} gives only the finiteness of the intersection of curves in $\G_m^{\ell}$ with algebraic subgroups. As already mentioned after Theorem~\ref{thm:genAR2}, it is of high interest to have available uniform bounds for the size of this intersection. This  implies uniform bounds on the degree of $h$ in Theorem~\ref{thm:genAR2}.

More generally, one can ask  for the number of solutions to $f(x,y) = 0$, with $x^n , y^m \in S$ for some nonzero integers $n$ and $m$, where $S$ is the group of $S$-units of some fixed number field. This  
leads again to further generalisations.

It is certainly interesting to obtain a similar result as Theorem~\ref{thm:multivAR} for  
$$
\gcd\(H_1\(F_1^{n_1},\ldots,F_{s}^{n_{s}}\),H_2\(G_1^{m_1},\ldots,G_{r}^{m_{r}}\)\),
$$
with polynomials $H_1\in\C[Y_1,\ldots,Y_s]$, $H_2\in\C[Z_1,\ldots,Z_r]$ and also $F_1,\ldots,F_s,G_1,\ldots,G_r\in\C[X_1,\ldots,X_{\ell}]$.

If one chooses 
$$H_1=Y_1\cdots Y_s-1,\quad H_2=Z_1\cdots Z_r-1,$$ 
then following the same proof as for Theorem~\ref{thm:multivAR}, we reduce (via specialisations) the problem to Theorem~\ref{thm:genAR2}, and thus get that 
\begin{equation}
\label{eq:degmultiv}
\deg \gcd\(H_1\(F_1^{n_1},\ldots,F_{s}^{n_{s}}\),H_2\(G_1^{m_1},\ldots,G_{r}^{m_{r}}\)\)
\end{equation}
is bounded by a constant depending only on $F_1,\ldots,F_s,G_1,\ldots,G_r$.

However, the approach of Theorem~\ref{thm:multivAR} does not seem to work for more general 
multivariate polynomials $H_1,H_2$.

\subsection{Dynamical analogues} 

Another interesting direction of research  is obtaining
dynamical  analogues of the results of Ailon-Rudnick~\cite{AR} and
Silverman~\cite{Si}. That is, investigating 
the greatest common divisors of polynomials   iterates.

More precisely, let $\K$ be a field and $f,g\in\K[X]$. We define
$$
f^{(0)}=T,\quad f^{(n)}=f(f^{(n-1)}),\quad  n\ge 1,
$$
and similarly for $g$. 

\begin{problem}
Give, under some natural conditions, 
an upper bound for
$$
\deg\gcd\(f^{(n)},g^{(m)}\).
$$
\end{problem}

\begin{problem}
\label{prob:gcd 1}
Show, under some natural conditions,  
that the iterates of $f$ and $g$ are coprime for infinitely many $n,m$.
\end{problem}
 
We note that some conditions on $f,g$ are certainly needed in 
Problem~\ref{prob:gcd 1} as, 
for example, 
if $f$ and $g$  have $0$ as fixed point, that is, $f(0) = g(0) = 0$, 
then
 $f^{(n)}$ and $g^{(m)}$ are never coprime.

We note that there are many results regarding the arithmetic structure of polynomial iterates. For example in~\cite{GNOS,GOS,Jon,JB} and references therein, results regarding the irreducibility of iterates are given. 
Irreducible polynomials $f\in\Q[X]$ such that all the iterates $f^{(n)}$, $n\ge 1$, remain irreducible are called {\it stable polynomials}. For quadratic polynomials the stability is given by the presence of squares in the orbit of the critical point of the polynomial.
Thus, if $f,g\in\K[X]$ are stable, then $f^{(n)}$ and $g^{(m)}$ are coprime for every $n,m\ge 1$.

For $h_1,h_2\in\K[X]$, one can also consider the more general case
$$
G_{n,m}=\gcd\(h_1\(f^{(n)}\),h_2\(g^{(m)}\)\).
$$

We note that, following the ideas of~\cite[Theorem 1]{AR} and of this paper, bounding the zeros of $G_{n,m}$ reduces to proving the finiteness (or even finding uniform bounds) of the number of $t\in\C$ such that $(f^{(n)}(t),g^{(m)}(t))\in V$, where $V$ is  the set of zeros of $\{h_1(X_1),h_2(X_2)\}$.

This naturally leads to the question of   counting the occurrences 
$$
\(f^{(n)}(t),g^{(m)}(t)\) \in V, \qquad (n, m, t) \in [1, N] \times [1,N] \times \C,
$$
for an arbitrary variety $V \subseteq \C^2$ and 
a sufficiently large  integer $N\ge 1$.

For a fixed $t$ and the diagonal case $n=m$, this is of the same  flavour as the uniform 
{\it dynamical Mordell--Lang conjecture\/}, which, 
for a fixed $(t_1,t_2)\in\C^2$ asserts that the iterates $n,m\ge 1$ such that 
$\(f^{(n)}(t_1),g^{(m)}(t_2)\)\in V$, see~\cite{BGKT,GT,GTZ1} and references therein, lie in finitely many arithmetic progressions (which number  does not depend on $t_1,t_2$).

\section*{Acknowledgements}

The author is very grateful to Joseph Silverman for drawing the attention on the Ailon-Rudnick Theorem and related results. The author would also like to thank Igor Shparlinski, Joseph Silverman, Thomas Tucker and Umberto Zannier for their valuable suggestions and stimulating discussions, and also for their comments on an early version of the paper. The research of A.~O. was supported by the 
UNSW Vice Chancellor's Fellowship.

\end{document}